\documentclass[11pt]{amsart}
\usepackage{esint}
\usepackage{hyperref}
\usepackage{color}
\usepackage{graphicx}
\usepackage{verbatim}
\usepackage{enumitem}
\usepackage{mathtools}

\newtheorem{theorem}{Theorem}[section]
\newtheorem{lemma}[theorem]{Lemma}
\newtheorem{proposition}[theorem]{Proposition}
\newtheorem{corollary}[theorem]{Corollary}

\theoremstyle{remark}
\newtheorem{remark}[theorem]{Remark}

\numberwithin{equation}{section}
\providecommand{\Ric}{\operatorname{Ric}}
\providecommand{\Sec}{\operatorname{Sec}}

\providecommand{\ind}{\operatorname{Ind}}
\newcommand{\II}{\mathrm I\!\mathrm I}

\DeclareMathOperator*{\essinf}{ess\,inf}

\title[Second Robin eigenvalue bounds for Schr\"odinger operators]
{Second Robin eigenvalue bounds for Schr\"odinger operators on Riemannian surfaces}
\author[R.\ Antonia]{Railane Antonia}
\author[M.\ P.\ Cavalcante]{Marcos P.\ Cavalcante}
\author[V.\ Souza]{Vinicius Souza}

\address{Institute of Mathematics, Federal University of Alagoas, Macei\'o, Brazil}
\email{railane.silva@academico.ufpb.br}
\email{marcos@pos.mat.ufal.br}
\email{vinicius.souza@im.ufal.br}

\date{\today}

\begin{document}

\begin{abstract}
Let $(\Sigma^2,ds^2)$ be a compact Riemannian surface, possibly with boundary, and
consider Schr\"odinger-type operators of the form $L=\Delta+V-aK$ together with
natural Robin and Steklov-type boundary conditions incorporating a boundary
potential $W$ and (in the curvature-corrected setting) the geodesic curvature
$\kappa_g$ of $\partial\Sigma$. Our main contribution is a geometric upper bound
for the second Robin eigenvalue in terms of the topology of $\Sigma$ and the
integrals of $V$ and $W$, obtained via a Hersch balancing argument on the capped
surface. As a geometric application, we derive sharp topological restrictions for
compact two-sided free boundary minimal surfaces of Morse index at most one inside
geodesic balls of negatively curved pinched Cartan--Hadamard $3$-manifolds under a
mild radius condition. We also prove complementary upper bounds for first
eigenvalues in the closed and Robin settings, including rigidity in the
curvature-corrected case, and we establish Steklov-type estimates in a coercive
regime where the Dirichlet-to-Neumann operator is well defined for all boundary
data.
\end{abstract}

\maketitle

\section{Introduction}\label{sec:intro}

Let $(\Sigma^2,ds^2)$ be a connected compact Riemannian surface, possibly with
smooth boundary $\partial\Sigma$. Given $V\in L^\infty(\Sigma)$ and $a\in\mathbb R$,
we consider the Schr\"odinger-type operator
\begin{equation}\label{eq:Schrodinger}
L=\Delta+V-aK,
\end{equation}
where $\Delta=\mathrm{div}\,\nabla\le 0$ is the Laplace--Beltrami operator and $K$
is the Gaussian curvature of $\Sigma$. When $\partial\Sigma\neq\emptyset$ we also
allow a boundary potential $W\in L^\infty(\partial\Sigma)$ and a boundary operator
\begin{equation}\label{eq:boundary-operator}
Bu=\frac{\partial u}{\partial\nu}-Wu+a\,\kappa_g\,u,
\end{equation}
where $\nu$ denotes the outward unit conormal and $\kappa_g$ is the geodesic
curvature of $\partial\Sigma$ in $\Sigma$.

\medskip
 
Spectral estimates for operators of the form \eqref{eq:Schrodinger} (and of the associated boundary problems below) are closely
linked to geometry. A fundamental example is the Jacobi operator governing the
second variation of area for minimal and constant mean curvature surfaces, which
can be written in Schr\"odinger form. This connection has long been exploited to
derive quantitative constraints on stability, index, and topology in surface
theory; see for instance \cite{
BerardCastillon,
Castillon,
Silveira,
dCPeng,
Espinar,
EspinarRosen,
FCSchoen,
Ros25,
SY79} and the references therein.

\medskip

For Schr\"odinger-type operators on closed surfaces, there is a large literature on
upper bounds for eigenvalues in terms of conformal/topological data, going back
to classical ideas of Hersch and the subsequent development of conformal volume
and branched-cover methods; see e.g.\ \cite{Hersch,Yau1987,ElSoufiIlias1992,R06}.
In the presence of boundary, Robin and Steklov spectra encode both interior and boundary geometry, and sharp inequalities often require combining conformal test functions with boundary trace estimates; see for instance \cite{ColboisGirouardGordonSher,Escobar,FraserSchoen11,FraserSchoen13}
and also \cite{LevitinMangoubiPolterovich} for related developments.

\medskip

Our first main theorem is an upper bound for the \emph{second} eigenvalue of the
Robin problem for the general Schr\"odinger operator $L_0=\Delta+V$ (i.e.\ $a=0$):
\begin{equation}\label{eq:Robin-intro}
\left\{
\begin{aligned}
\Delta u+Vu+\mu u&=0 &&\text{in }\Sigma,\\
\partial_\nu u&=Wu &&\text{on }\partial\Sigma.
\end{aligned}
\right.
\end{equation}
We index Robin eigenvalues as $\mu_0<\mu_1\le \mu_2\le\cdots$, so that $\mu_0$ denotes the first eigenvalue.

\begin{theorem}\label{thm:intro-mu1}
Let $(\Sigma^2,ds^2)$ be a compact oriented surface with boundary, of genus $\gamma$.
For \eqref{eq:Robin-intro} with $V\in L^\infty(\Sigma)$ and $W\in L^\infty(\partial\Sigma)$,
\begin{equation}\label{eq:intro-mu1}
\mu_1\,|\Sigma|
\;<\;
8\pi\Big\lfloor \frac{\gamma+3}{2}\Big\rfloor
\;-\;\int_\Sigma V\;-\;\int_{\partial\Sigma}W.
\end{equation}
\end{theorem}

The estimate \eqref{eq:intro-mu1} can be viewed as a boundary counterpart of
the conformal topological eigenvalue bounds of El Soufi--Ilias \cite{ElSoufiIlias1992} on closed
manifolds (and related works), now adapted to the Robin problem on surfaces with
boundary.
The proof combines three ingredients: (i)~we cap $\partial\Sigma$ by discs to
obtain a closed surface $\widehat\Sigma$; (ii)~we use a conformal branched cover
$\widehat\Sigma\to\mathbb S^2$ of controlled degree (cf.\ \cite{RosVergasta1995, ChenFraserPang}); and (iii)~we apply a Hersch balancing argument \cite{Hersch} to build conformal test functions,
together with the min--max characterization of $\mu_1$. 

\medskip

Our second main theorem gives topological restrictions for compact two-sided free
boundary minimal surfaces of low Morse index inside geodesic balls of negatively
curved pinched Cartan--Hadamard $3$-manifolds.  This establishes a hyperbolic analogue of the index--one topological bound proved by Chen--Fraser--Pang in ambients with $\Ric\ge0$ and
weakly convex boundary~\cite{ChenFraserPang} (see also \cite{Longa2022}). Such restrictions are part of a
broader circle of ideas relating index bounds, curvature pinching, and topology
for minimal/CMC surfaces; see for example 
\cite{ACS,CO, MazetMendes,Nunes2017,RosVergasta1995} and
the references therein.

Let $M^3$ be Cartan--Hadamard with pinched sectional curvature
\[
-\kappa_1\le \Sec_M\le -\kappa_2<0,\qquad 0<\kappa_2\le\kappa_1.
\]
Let $B_o(R)\subset M^3$ be a geodesic ball. If $\Sigma^2\subset B_o(R)$ is a compact
two-sided free boundary minimal surface, then its Jacobi operator is $J=\Delta+V$
with $V=\Ric_M(N,N)+|A|^2$ and Robin boundary term $W=\II^{\partial B_o(R)}(N,N)$.
Since $\ind(\Sigma)\le 1$ implies $\mu_1(J)\ge 0$, Theorem~\ref{thm:intro-mu1}
yields restrictions once $V$ and $W$ are controlled by curvature comparison.

\begin{theorem}\label{thm:intro-FBMS}
Let $M^3$ be Cartan--Hadamard manifold with $-\kappa_1\le \Sec_M\le -\kappa_2<0$ and fix $o\in M$,
$R>0$. Assume
\begin{equation}\label{eq:intro-Rcond}
4\kappa_{1}\tanh\!\Big(\frac{\sqrt{\kappa_{2}}\,R}{2}\Big)
\ \le\
3\kappa_{2}\coth(\sqrt{\kappa_{2}}\,R).
\end{equation}
If $\Sigma^{2}\subset B_{o}(R)$ is a compact two-sided free boundary minimal surface with
$\ind(\Sigma)\le 1$, then the only possibilities are
\[
\gamma\in\{0,1\}\ \text{ with }\ r\le 3,
\qquad\text{or}\qquad
\gamma\in\{2,3\}\ \text{ with }\ r=1,
\]
where $\gamma$ is the genus of $\Sigma$ and $r$ is the number of boundary components.
\end{theorem}

\begin{remark}
The left-hand side of \eqref{eq:intro-Rcond} increases with $R$ and the right-hand side
decreases with $R$, so \eqref{eq:intro-Rcond} is a genuine geometric restriction linking
$R$ and the pinching ratio $\kappa_1/\kappa_2$. We discuss its role and provide a
qualitative interpretation in Section~\ref{sec:FBMS}.
\end{remark}

\medskip

We also establish upper bounds for first eigenvalues in the closed and Robin
settings, including rigidity in the curvature-corrected case (compare with
intrinsic counterparts of results in immersed surface theory such as
\cite{AliasMO,BatistaS}). Finally, we obtain Steklov-type estimates in a coercive
regime where the Dirichlet problem is well posed for all boundary data and the
Dirichlet-to-Neumann operator is defined on the full trace space; cf.\ related
spectral-geometric frameworks in \cite{AEKS,ColboisGirouardGordonSher,Escobar,FraserSchoen13,Tran,JZhu}.

\medskip

\noindent\textbf{Organization.}
Section~\ref{sec:closed} contains the closed-surface first-eigenvalue bound and a
stability corollary. Section~\ref{sec:Robin} proves the first Robin bound and then
establishes Theorem~\ref{thm:intro-mu1}. Section~\ref{sec:FBMS} proves
Theorem~\ref{thm:intro-FBMS}. Section~\ref{sec:Steklov} develops the Jacobi--Steklov
framework and proves Steklov-type eigenvalue bounds.

\section{Closed surfaces: a first-eigenvalue bound}\label{sec:closed}

Let $\Sigma$ be a closed oriented surface and let $\lambda_0(L)$ denote the first
eigenvalue of $L$ in \eqref{eq:Schrodinger}. It is classical that $\lambda_0(L)$ is
simple and admits a positive eigenfunction.

\begin{theorem}\label{thm:closed-lambda0}
Let $(\Sigma^2,ds^2)$ be a closed oriented surface, and let $L=\Delta+V-aK$ with
$V\in L^\infty(\Sigma)$ and $a\in\mathbb R$. Then
\[
\lambda_0(L)\,|\Sigma|
\;\le\;
-(\essinf_{\Sigma}V)\,|\Sigma|+2\pi a\,\chi(\Sigma),
\]
where $|\Sigma|$ is the area and $\chi(\Sigma)=2-2\gamma$ is the Euler characteristic.
Moreover, if $a\neq 0$ and equality holds, then $V$ and $K$ are constant and
\[
K=\frac{\lambda_0(L)+V}{a}.
\]
\end{theorem}

\begin{proof}
Let $u>0$ solve $Lu+\lambda_0(L)u=0$. By elliptic regularity $u\in C^\infty(\Sigma)$,
so $\ln u\in C^\infty(\Sigma)$. As in Perdomo \cite{Perdomo} (see also \cite{AliasMO}),
\[
\Delta\ln u
=\frac{\Delta u}{u}-\frac{|\nabla u|^2}{u^2}
=-\lambda_0(L)-V+aK-|\nabla\ln u|^2.
\]
Integrating over $\Sigma$ and using $\int_\Sigma \Delta\ln u=0$, we get
\begin{equation}\label{eq:closed-identity}
\lambda_0(L)|\Sigma|
=
-\int_\Sigma V
+a\int_\Sigma K
-\int_\Sigma |\nabla\ln u|^2.
\end{equation}
Since $V\ge \essinf_\Sigma V$ a.e.\ and $\int_\Sigma K=2\pi\chi(\Sigma)$ by Gauss--Bonnet,
we obtain the desired inequality.

If equality holds and $a\neq 0$, then \eqref{eq:closed-identity} forces $\nabla u\equiv 0$,
so $u$ is constant, hence $V\equiv \inf_\Sigma V$ and $K$ is constant, yielding the formula.
\end{proof}

We say that the Schr\"odinger operator $L$ is \emph{stable} on a closed surface if $-L\ge 0$ in the $L^2$ sense, or equivalently $\lambda_0(L)\ge 0$. We have the following consequence. 

\begin{corollary}\label{cor:closed-stable}
Let $\Sigma$ be a closed oriented surface. If $L=\Delta+V-aK$ is stable on $\Sigma$, then:
\begin{enumerate}[label=\emph{(\arabic*)}, itemsep=2pt]
\item for $a=0$, $\displaystyle\int_\Sigma V\le 0$;
\item for $a>0$ and $V\ge 0$ a.e.\ on $\Sigma$, $\Sigma$ is either a sphere or a torus; in the torus case
$V\equiv 0$ a.e.
\end{enumerate}
\end{corollary}

\begin{proof}
If $a=0$, from \eqref{eq:closed-identity} and $\lambda_0(L)\ge 0$ we obtain
$0\le -\int_\Sigma V-\int_\Sigma|\nabla\ln u|^2$, hence $\int_\Sigma V\le 0$.

If $a>0$ and $V\ge 0$, Theorem~\ref{thm:closed-lambda0} implies
$0\le \lambda_0(L)|\Sigma|\le 2\pi a\,\chi(\Sigma)$, so $\chi(\Sigma)\ge 0$, hence $\Sigma$
is a sphere or a torus. If $\chi(\Sigma)=0$ (torus), then the inequality forces $\inf_\Sigma V=0$.
Since also $\int_\Sigma V\le a\int_\Sigma K=0$ from \eqref{eq:closed-identity}, we get $V\equiv 0$  a.e.
\end{proof}

These estimates may be seen as intrinsic counterparts of results proved in the immersed setting,
compare \cite[Theorem 2.1]{AliasMO} and \cite[Theorem 1.1]{BatistaS}.

\section{Surfaces with boundary: the Robin problem}\label{sec:Robin}

Throughout this section, $\Sigma$ is a compact oriented surface with smooth boundary.
Its Euler characteristic is
\[
\chi(\Sigma)=2-2\gamma-r,
\]
where $\gamma$ is the genus and $r$ is the number of boundary components.

\subsection{The first Robin eigenvalue}\label{subsec:Robin-mu0}

Consider the Robin problem for $L=\Delta+V-aK$ with boundary operator \eqref{eq:boundary-operator}:
\begin{equation}\label{eq:Robin-curv}
\left\{
\begin{aligned}
Lu+\mu u&=0 &&\text{in }\Sigma,\\
Bu&=0 &&\text{on }\partial\Sigma.
\end{aligned}
\right.
\end{equation}
Its spectrum is discrete:
\[
\mu_0(L)<\mu_1(L)\le\mu_2(L)\le\cdots\to+\infty
\]
And the associated quadratic form on $H^1(\Sigma)$ is
\begin{equation}\label{eq:Q-curv}
Q[u]=\int_\Sigma |\nabla u|^2-\int_\Sigma (V-aK)u^2-\int_{\partial\Sigma}(W-a\kappa_g)u^2.
\end{equation}
The first eigenvalue satisfies the Rayleigh characterization
\begin{equation}\label{eq:Rayleigh-mu0}
\mu_0(L)=\inf_{u\in H^1(\Sigma)\setminus\{0\}}\frac{Q[u]}{\int_\Sigma u^2}.
\end{equation}

Our first estimate with boundary controls the first Robin eigenvalue $\mu_0(L)$
in terms of geometric and topological data of $\Sigma$, in the spirit of
Theorem~\ref{thm:closed-lambda0}.

\begin{theorem}\label{thm:Robin-mu0}
Let $(\Sigma^2,ds^2)$ be a compact oriented surface with boundary, and let $\mu_0(L)$ be the
first eigenvalue of \eqref{eq:Robin-curv}. Then
\[
\mu_0(L)\,|\Sigma|
\;\le\;
-(\essinf_\Sigma V)\,|\Sigma|
-(\essinf_{\partial\Sigma}W)\,|\partial\Sigma|
+2\pi a\,\chi(\Sigma).
\]
Moreover, if $a\neq 0$ and equality holds, then $V$ and $W$ are constant a.e.\ and
\[
K=\frac{\mu_0(L)+V}{a},\qquad \kappa_g=\frac{W}{a}
\]
are constant.
\end{theorem}

\begin{proof}
Test \eqref{eq:Rayleigh-mu0} with $u\equiv 1$. Since
$V\ge \essinf_\Sigma V$ a.e.\ and $W\ge \essinf_{\partial\Sigma}W$ a.e.\ on $\partial\Sigma$,
and by Gauss--Bonnet,
\[
\int_\Sigma K+\int_{\partial\Sigma}\kappa_g=2\pi\chi(\Sigma),
\]
we obtain the stated inequality.

If equality holds, then $u\equiv 1$ is a first eigenfunction (cf. \cite[Theorem~2.4]{BCM2024}),
hence $L1+\mu_0(L)=0$ and $B1=0$, which imply $V$ and $W$ are constant; if $a\neq 0$ this forces
$K$ and $\kappa_g$ to be constant as well.
\end{proof}

\begin{corollary} \label{cor:Robin-stable}
Assume $\mu_0(L)\ge 0$ in \eqref{eq:Robin-curv}. If $a>0$, $V\ge 0$ on $\Sigma$, and $W\ge 0$
on $\partial\Sigma$, then $\Sigma$ is either a disk or an annulus.
\end{corollary}

\begin{proof}
From Theorem~\ref{thm:Robin-mu0},
$0\le \mu_0(L)|\Sigma|\le 2\pi a\,\chi(\Sigma)$, hence $\chi(\Sigma)\ge 0$. Since
$\chi(\Sigma)=2-2\gamma-r$ with $r\ge 1$, we must have $\gamma=0$ and $r\in\{1,2\}$.
\end{proof}

\subsection{The second Robin eigenvalue: proof of Theorem~\ref{thm:intro-mu1}}\label{subsec:Robin-mu1}

We now specialize to $a=0$, i.e.\ $L_0=\Delta+V$, and consider the Robin problem
\eqref{eq:Robin-intro}. Its quadratic form is
\begin{equation}\label{eq:Q0}
Q[u]=\int_\Sigma |\nabla u|^2-\int_\Sigma V u^2-\int_{\partial\Sigma}W u^2.
\end{equation}
Let $u_0>0$ be a first eigenfunction, and recall the min--max characterization
\begin{equation}\label{eq:Rayleigh-mu1}
\mu_1
=\inf\left\{\frac{Q[u]}{\int_\Sigma u^2}\,:\,u\in H^1(\Sigma)\setminus\{0\},\
\int_\Sigma u\,u_0=0\right\}.
\end{equation}

\medskip 

Let $\widehat\Sigma$ be the closed surface obtained by conformally capping each boundary
component of $\Sigma$ by a disc. Then $\widehat\Sigma$ has the same genus $\gamma$ as $\Sigma$.
A Hersch-type balancing argument on $\widehat\Sigma$ yields the following lemma proved by Chen, Fraser and Pang \cite{ChenFraserPang}.

\begin{lemma}[Chen--Fraser--Pang]\label{lem:hersch}
Let $u_0>0$ be a first eigenfunction on $\Sigma$. There exists a conformal map
$\Phi:\widehat\Sigma\to \mathbb S^2\subset\mathbb R^3$ such that
\[
\int_\Sigma \Phi\,u_0=0\in\mathbb R^3
\qquad\text{and}\qquad
\deg(\Phi)\le \Big\lfloor\frac{\gamma+3}{2}\Big\rfloor.
\]
\end{lemma}

Write $\Phi=(\phi_1,\phi_2,\phi_3)$. Then $\sum_{i=1}^3\phi_i^2\equiv 1$ on $\Sigma$, and
\begin{equation}\label{eq:Dir-energy}
\sum_{i=1}^3\int_{\Sigma}|\nabla\phi_i|^2
<
\sum_{i=1}^3\int_{\widehat\Sigma}|\nabla\phi_i|^2
\le
8\pi\Big\lfloor\frac{\gamma+3}{2}\Big\rfloor,
\end{equation}
since the Dirichlet energy of a conformal map to $\mathbb S^2$ equals $2$ times the area of the
image, hence $8\pi$ times the degree, and restriction from $\widehat\Sigma$ to $\Sigma$ strictly
decreases energy unless $\partial\Sigma=\emptyset$.

Moreover, Lemma~\ref{lem:hersch} gives the orthogonality conditions
\begin{equation}\label{eq:balance}
\int_\Sigma u_0\,\phi_i=0,\qquad i=1,2,3.
\end{equation}
Thus each $\phi_i$ is admissible in \eqref{eq:Rayleigh-mu1}, and
\begin{equation}\label{eq:mu1-phi}
\mu_1\int_\Sigma \phi_i^2
\le
\int_\Sigma|\nabla\phi_i|^2-\int_\Sigma V\phi_i^2-\int_{\partial\Sigma}W\phi_i^2.
\end{equation}

\begin{theorem}\label{thm:upper-bound-mu1}
Let $(\Sigma^2,ds^2)$ be a compact oriented surface with boundary of genus $\gamma$, and let
$\mu_1$ be the second eigenvalue of \eqref{eq:Robin-intro}. Then
\[
\mu_1\,|\Sigma|
\;<\;
8\pi\Big\lfloor\frac{\gamma+3}{2}\Big\rfloor
\;-\;\int_\Sigma V\;-\;\int_{\partial\Sigma}W.
\]
\end{theorem}

\begin{proof}
Sum \eqref{eq:mu1-phi} over $i=1,2,3$ and use $\sum_i\phi_i^2\equiv 1$:
\[
\mu_1|\Sigma|
\le
\sum_{i=1}^3\int_\Sigma|\nabla\phi_i|^2
-\int_\Sigma V
-\int_{\partial\Sigma}W.
\]
Apply \eqref{eq:Dir-energy} to conclude the strict inequality.
\end{proof}

\subsubsection{The case of planar domains}\label{subsubsec:planar-domains}

It is natural to seek upper estimates for $\mu_1$ that also reflect the number $r$ of boundary components.
To incorporate $r$, we adapt a hemisphere mapping in the spirit of Nunes' modified Hersch argument for free-boundary stable CMC surfaces (see \cite{Nunes2017}). 
By a result of Gabard \cite{Gabard2006}, there exists a proper conformal branched cover
\[
\Phi:\Sigma\longrightarrow \mathbb{D}\cong\mathbb{S}^2_+\subset\mathbb{R}^3,
\qquad \deg\Phi\le \gamma+r,
\]
and, after applying a conformal automorphism of $\mathbb{S}^2_+$ (see also \cite{Mendes2018}), we obtain the balancing conditions
\begin{equation}\label{eq:balance-hemisphere}
\int_\Sigma u_0\,\phi_i=0,\qquad i=1,2,
\end{equation}
where $\Phi=(\phi_1,\phi_2,\phi_3)$ with $\phi_1^2+\phi_2^2+\phi_3^2=1$ on $\Sigma$ and $\phi_3=0$ on $\partial\Sigma$.

Thus $\phi_1,\phi_2$ are admissible in \eqref{eq:Rayleigh-mu1}, yielding
\begin{equation}\label{eq:mu1-phi12-hemisphere}
\mu_1\int_\Sigma \phi_i^2
\;\le\;
\int_\Sigma|\nabla\phi_i|^2
-\int_\Sigma V\,\phi_i^2
-\int_{\partial\Sigma}W\,\phi_i^2,
\qquad i=1,2.
\end{equation}

The Dirichlet energy is controlled by
\begin{equation}\label{energy-hemisphere}
\sum_{i=1}^3\int_\Sigma|\nabla\phi_i|^2\;\le\;4\pi(\gamma+r).
\end{equation}

To also exploit $\phi_3$, which vanishes on $\partial\Sigma$, consider the Dirichlet problem
\begin{equation}\label{eq:Dirichlet}
\left\{
\begin{aligned}
\Delta u+V u+\lambda^D u&=0&&\text{in }\Sigma,\\
u&=0&&\text{on }\partial\Sigma,
\end{aligned}
\right.
\end{equation}
and denote its first eigenvalue by $\lambda_0^D(L_0)$, or simply by $\lambda_0^D.$

\medskip

In the case of domains in Euclidean space (thus $\gamma=0$), Filonov’s method \cite{Filonov} yields, in particular, the comparison
\[
\mu_1 \;<\; \lambda_0^D
\]
for the Euclidean Laplacian with $V\equiv 0$ and $W$ with nonnegative average, which is the situation we use below. 

\begin{theorem}\label{thm:upper-bound-mu1-boundary}
Assume $\Sigma\subset\mathbb{R}^2$ is a bounded planar domain with $r$ boundary components, $V\equiv 0$, and $\int_{\partial\Sigma}W\ge 0$. Then
\[
\mu_1\,|\Sigma|
\;<\;
4\pi r\;-\;\int_{\partial\Sigma}W.
\]
\end{theorem}

\begin{proof}
Using $\mu_1< \lambda_0^D$ and that $\phi_3=0$ on $\partial\Sigma$, we have
\[
\mu_1\int_\Sigma\phi_3^2
\;<\;
\int_\Sigma|\nabla\phi_3|^2.
\]
Adding this to \eqref{eq:mu1-phi12-hemisphere} for $i=1,2$ and using $\phi_1^2+\phi_2^2=1$ on $\partial\Sigma$, we obtain
\[
\mu_1\int_\Sigma(\phi_1^2+\phi_2^2+\phi_3^2)
\;<\;
\sum_{i=1}^3\! \int_\Sigma|\nabla\phi_i|^2
-\int_{\partial\Sigma}W.
\]
Since $\phi_1^2+\phi_2^2+\phi_3^2\equiv 1$, the claim follows from \eqref{energy-hemisphere} with $\gamma=0$.
\end{proof}

\begin{remark}\label{rem:FriedlanderFilonov}
The key extra input in the planar estimate above is the (strict) comparison
$\mu_1<\lambda_0^D$ (in the Euclidean case, under the sign assumption on $W$),
which allows us to exploit the third coordinate $\phi_3$ that vanishes on
$\partial\Sigma$. Any extension of such a comparison to more general settings
(e.g.\ domains in Riemannian surfaces and/or Schr\"odinger operators with
$V\not\equiv 0$) would yield corresponding refinements of
Theorem~\ref{thm:upper-bound-mu1-boundary}.
Inequalities comparing Dirichlet and Neumann spectra of Friedlander type and
their variants on Riemannian manifolds have been studied by Friedlander
\cite{Friedlander}, Arendt--Mazzeo \cite{ArendtMazzeo}, Rohleder \cite{Rohleder2021} and Ilias--Shouman
\cite{IliasShouman}, among others.
\end{remark}

\section{Application to free boundary minimal surfaces}\label{sec:FBMS}

In this section we prove Theorem~\ref{thm:intro-FBMS}. Let $M^3$ be complete and let
$B_o(R)\subset M$ be a geodesic ball of radius $R$ centered at $o$. Let
$\Sigma^2\subset B_o(R)$ be a compact two-sided free boundary minimal surface:
\[
\partial\Sigma\subset \partial B_o(R),
\qquad
\Sigma\ \text{meets }\partial B_o(R)\ \text{orthogonally along }\partial\Sigma.
\]
Let $N$ be a global unit normal, $A$ the shape operator of $\Sigma^2$, and
$\II^{\partial B_o(R)}$ the second fundamental form of $\partial B_o(R)$ with respect
to the outward unit normal. The index form is
\[
\mathcal I(u,u)
=\int_\Sigma |\nabla u|^2
-\int_\Sigma\big(\Ric_{M}(N,N)+|A|^2\big)u^2
-\int_{\partial\Sigma}\II^{\partial B_o(R)}(N,N)\,u^2.
\]

The quadratic form $\mathcal I$ is precisely the Robin form associated to the Jacobi operator
$J=\Delta+V$ with boundary term $W$, where
\[
V=\Ric_M(N,N)+|A|^2,
\qquad
W=\II^{\partial B_o(R)}(N,N).
\]
Let $\{\mu_k(J)\}_{k\ge 0}$ be the corresponding Robin eigenvalues, ordered as
$\mu_0(J)<\mu_1(J)\le\mu_2(J)\le\cdots$. We define the Morse index $\ind(\Sigma)$ as the
number of negative eigenvalues (counted with multiplicity). In particular,
\[
\ind(\Sigma)\le 1 \quad\Longleftrightarrow\quad \mu_1(J)\ge 0.
\]

\medskip

\subsection{Curvature consequences under pinching.}
Assume $M$ is a Cartan--Hadamard manifold satisfying 
\[
-\kappa_1\le \Sec_M\le -\kappa_2<0,\qquad 0<\kappa_2\le \kappa_1.
\]
Then for any unit vector $N$,
\[
-2\kappa_1\le \Ric_M(N,N)\le -2\kappa_2,
\]
hence
\[
V\ge |A|^2-2\kappa_1.
\]
Moreover, by Hessian comparison under $\Sec_M\le -\kappa_2$, the geodesic sphere
$\partial B_o(R)$ is strictly convex and satisfies
\[
\II^{\partial B_o(R)}\ \ge\ \sqrt{\kappa_2}\,\coth(\sqrt{\kappa_2}\,R)\,g_{\partial B_o(R)}.
\]
Since $N\in T(\partial B_o(R))$ along $\partial\Sigma$ by the free boundary condition,
\begin{equation}\label{eq:W-lower}
W=\II^{\partial B_o(R)}(N,N)\ \ge\ \sqrt{\kappa_2}\,\coth(\sqrt{\kappa_2}\,R)>0.
\end{equation}
Likewise, the geodesic curvature $\kappa_g$ of $\partial\Sigma$ in $\Sigma$ satisfies
\begin{equation}\label{eq:kg-lower}
\kappa_g=\II^{\partial B_o(R)}(T,T)\ \ge\ \sqrt{\kappa_{2}}\,\coth(\sqrt{\kappa_{2}}\,R),
\end{equation}
where $T$ is the unit tangent to $\partial\Sigma\subset\partial B_o(R)$

Finally, for a minimal surface ($H\equiv 0$), the Gauss equation gives
\[
K=\Sec_M(T\Sigma)-\frac{|A|^2}{2},
\]
so using $\Sec_M(T\Sigma)\ge -\kappa_1$ we obtain
\begin{equation}\label{eq:A2-vs-K}
-|A|^2=2K-2\Sec_M(T\Sigma)\ \le\ 2K+2\kappa_1.
\end{equation}

\begin{theorem}\label{thm:FBMS}
Let $M^{3}$ be a Cartan--Hadamard manifold whose sectional curvature satisfies
\[
-\kappa_{1}\le \Sec_{M}\le -\kappa_{2}<0,\qquad 0<\kappa_{2}\le \kappa_{1}.
\]
Fix $o\in M$ and $R>0$, and assume
\begin{equation}\label{eq:R-condition}
4\kappa_{1}\tanh\!\Big(\frac{\sqrt{\kappa_{2}}\,R}{2}\Big)
\ \le\
3\kappa_{2}\coth(\sqrt{\kappa_{2}}\,R).
\end{equation}
Let $\Sigma^{2}\subset B_{o}(R)$ be a compact two-sided free boundary minimal surface with
$\ind(\Sigma)\le 1$. Then the only possibilities are
\[
\gamma\in\{0,1\}\ \text{ with }\ r\le 3,
\qquad\text{or}\qquad
\gamma\in\{2,3\}\ \text{ with }\ r=1.
\]
\end{theorem}

\begin{proof}
Let $\mu_1$ be the second Robin eigenvalue of the Jacobi operator $J=\Delta+V$ with boundary
term $W$ as above. Since $\ind(\Sigma)\le 1$, we have $\mu_1\ge 0$. Theorem~\ref{thm:upper-bound-mu1}
gives
\[
0\le \mu_1|\Sigma|
<
8\pi\Big\lfloor\frac{\gamma+3}{2}\Big\rfloor
-\int_\Sigma V
-\int_{\partial\Sigma}W.
\]
Using $V\ge |A|^2-2\kappa_1$ and \eqref{eq:W-lower},
\begin{equation}\label{eq:FBMS-1}
0<
8\pi\Big\lfloor\frac{\gamma+3}{2}\Big\rfloor
-\int_\Sigma(|A|^2-2\kappa_1)
-\sqrt{\kappa_2}\coth(\sqrt{\kappa_2}R)\,|\partial\Sigma|.
\end{equation}

Next, from \eqref{eq:A2-vs-K},
\[
-\int_\Sigma(|A|^2-2\kappa_1)
=-\int_\Sigma|A|^2+2\kappa_1|\Sigma|
\le
2\int_\Sigma K+4\kappa_1|\Sigma|.
\]
Plugging into \eqref{eq:FBMS-1},
\begin{equation}\label{eq:FBMS-2}
0<
8\pi\Big\lfloor\frac{\gamma+3}{2}\Big\rfloor
+2\int_\Sigma K
+4\kappa_1|\Sigma|
-\sqrt{\kappa_2}\coth(\sqrt{\kappa_2}R)\,|\partial\Sigma|.
\end{equation}

By Gauss--Bonnet and \eqref{eq:kg-lower},
\[
\int_\Sigma K
=
2\pi(2-2\gamma-r)-\int_{\partial\Sigma}\kappa_g
\le
2\pi(2-2\gamma-r)-\sqrt{\kappa_2}\coth(\sqrt{\kappa_2}R)\,|\partial\Sigma|.
\]
Substituting into \eqref{eq:FBMS-2} yields
\begin{equation}\label{eq:FBMS-3}
0<
8\pi\Big\lfloor\frac{\gamma+3}{2}\Big\rfloor
+4\pi(2-2\gamma-r)
+4\kappa_1|\Sigma|
-3\sqrt{\kappa_2}\coth(\sqrt{\kappa_2}R)\,|\partial\Sigma|.
\end{equation}

Finally, apply the linear isoperimetric inequality of Lee--Seo \cite[Theorem~2.10(ii)]{LeeSeo2023}
(with $n=2$, $k=\sqrt{\kappa_2}$, and $H=0$):
\[
|\Sigma|\le \frac{1}{\sqrt{\kappa_2}}\tanh\!\Big(\frac{\sqrt{\kappa_2}R}{2}\Big)\,|\partial\Sigma|.
\]
Therefore,
\[
4\kappa_1|\Sigma|
\le
\frac{4\kappa_1}{\sqrt{\kappa_2}}\tanh\!\Big(\frac{\sqrt{\kappa_2}R}{2}\Big)\,|\partial\Sigma|.
\]
Substitute into \eqref{eq:FBMS-3} and use \eqref{eq:R-condition} to drop the resulting nonpositive
boundary-length term. We obtain
\[
2\Big\lfloor\frac{\gamma+3}{2}\Big\rfloor+2-2\gamma-r>0.
\]
Writing $\gamma=2m$ or $\gamma=2m+1$ shows
$2\lfloor(\gamma+3)/2\rfloor+2-2\gamma=4-2m$, hence $4-2m-r>0$, i.e.\ $r\le 3-2m$. Thus:
$m=0$ gives $\gamma\in\{0,1\}$ and $r\le 3$, while $m=1$ gives $\gamma\in\{2,3\}$ and $r\le 1$,
which completes the proof.
\end{proof}

\begin{remark}
Since $\tanh(\sqrt{\kappa_2}R/2)\nearrow 1$ and $\coth(\sqrt{\kappa_2}R)\searrow 1$ as $R\to\infty$,
the limiting form of \eqref{eq:R-condition} would be $4\kappa_1\le 3\kappa_2$.
Under our convention $\kappa_1\ge\kappa_2$, this cannot hold; hence \eqref{eq:R-condition} is
necessarily a \emph{small-radius} requirement. For $R\to 0$, one has
$\tanh(\sqrt{\kappa_2}R/2)\sim \sqrt{\kappa_2}R/2$ and $\coth(\sqrt{\kappa_2}R)\sim 1/(\sqrt{\kappa_2}R)$,
so \eqref{eq:R-condition} holds for $R$ sufficiently small (depending on $\kappa_1,\kappa_2$).
\end{remark}

\section{A Jacobi--Steklov framework}\label{sec:Steklov}

In this section we study the Steklov-type eigenvalue problem associated with the pair
$(L_0,B_0)$:
\begin{equation}\label{eq:Steklov}
\left\{
\begin{array}{rcll}
L_0u&=&0&\text{in }\Sigma,\\[1mm]
B_0u&=&\sigma u&\text{on }\partial\Sigma,
\end{array}
\right.
\qquad
L_0=\Delta+V,\quad B_0=\partial_\nu-W,
\end{equation}
where $\Sigma$ is a compact oriented surface with smooth boundary and
$V\in L^\infty(\Sigma)$, $W\in L^\infty(\partial\Sigma)$.
Related Steklov-type problems arising from free boundary minimal surface geometry
are discussed, for instance, by Lima--Menezes~\cite{LimaMenezes}.

\subsection{Coercivity and the Dirichlet-to-Neumann map}\label{subsec:DtN}

Unlike the classical Steklov problem (where $V\equiv 0$ and $W\equiv 0$), the Dirichlet problem for
$L_0$ may fail to be solvable for arbitrary boundary data. Consequently, the corresponding
Dirichlet-to-Neumann map may only be defined on a proper subspace of boundary traces. This issue is
analyzed in detail by Tran~\cite{Tran} in a slightly more general framework. To avoid such domain
restrictions, we work in a coercive regime that guarantees well-posedness for all boundary data,
namely positivity of the first Dirichlet eigenvalue.

Consider the Dirichlet problem
\begin{equation*}
\left\{
\begin{aligned}
\Delta u+Vu+\lambda^D u&=0&&\text{in }\Sigma,\\
u&=0&&\text{on }\partial\Sigma.
\end{aligned}
\right.
\end{equation*}
Let $\lambda_0^D(L_0)$ denote its first eigenvalue.

Assume henceforth that $\lambda_0^D(L_0)>0$. Then the interior quadratic form
\[
Q^D[u]:=\int_\Sigma |\nabla u|^2-\int_\Sigma Vu^2
\]
is coercive on $H_0^1(\Sigma)$. In particular, for each boundary datum
$h\in H^{1/2}(\partial\Sigma)$ there exists a unique minimizer of $Q^D$ among all
$H^1$-extensions of $h$. This minimizer is precisely the weak $L_0$-harmonic extension of $h$, and
it provides a canonical definition of the Dirichlet-to-Neumann map.

\begin{proposition}\label{prop:Dirichlet-wellposed}
Assume $\lambda_0^D(L_0)>0$. Given $h\in H^{1/2}(\partial\Sigma)$, consider the affine space
\[
\mathcal A_h:=\{u\in H^1(\Sigma): u|_{\partial\Sigma}=h\}.
\]
Then there exists a unique $\widehat h\in \mathcal A_h$ such that
\[
Q^D[\widehat h]=\min_{u\in \mathcal A_h} Q^D[u].
\]
Moreover, $\widehat h$ is the unique weak solution of $L_0u=0$ in $\Sigma$ with trace $h$.
\end{proposition}

\begin{proof}
Fix $h\in H^{1/2}(\partial\Sigma)$ and choose $u_h\in H^1(\Sigma)$ with $u_h|_{\partial\Sigma}=h$.
For $v\in H_0^1(\Sigma)$ set $u=u_h+v$, so $u\in\mathcal A_h$.

Since $\lambda_0^D(L_0)>0$, the form $Q^D$ is coercive on $H_0^1(\Sigma)$. Hence the functional
$v\mapsto Q^D[u_h+v]$ is strictly convex, coercive, and weakly lower semicontinuous on
$H_0^1(\Sigma)$, and therefore it admits a unique minimizer $v_*\in H_0^1(\Sigma)$.
Setting $\widehat h:=u_h+v_*$ yields the unique minimizer of $Q^D$ in $\mathcal A_h$.

The Euler--Lagrange equation gives, for every $\varphi\in H_0^1(\Sigma)$,
\[
\frac{d}{dt}\Big|_{t=0} Q^D[\widehat h+t\varphi]=0
\quad\Longleftrightarrow\quad
\int_\Sigma \langle \nabla \widehat h,\nabla \varphi\rangle-\int_\Sigma V\,\widehat h\,\varphi=0,
\]
which is exactly the weak formulation of $L_0\widehat h=0$ with trace $h$.
Uniqueness follows from coercivity.
\end{proof}

Under Proposition~\ref{prop:Dirichlet-wellposed}, we define the Dirichlet-to-Neumann map
$\Lambda_{L_0}:H^{1/2}(\partial\Sigma)\to H^{-1/2}(\partial\Sigma)$ by
\[
\Lambda_{L_0}(h):=\left.\partial_\nu\widehat h\right|_{\partial\Sigma},
\]
and the associated boundary operator
\[
\mathcal S:=\Lambda_{L_0}-W.
\]
Then \eqref{eq:Steklov} is equivalent to $\mathcal S h=\sigma h$ on $\partial\Sigma$, where
$h=u|_{\partial\Sigma}$. In the coercive regime $\lambda_0^D(L_0)>0$, the operator $\mathcal S$
defines a self-adjoint operator on $L^2(\partial\Sigma)$ with discrete spectrum, and its
eigenvalues coincide with the Steklov spectrum of \eqref{eq:Steklov} (cf.\ \cite{Tran}).

\subsection{Variational characterizations}\label{subsec:Steklov-variational}

Let
\[
Q[u]=\int_\Sigma |\nabla u|^2-\int_\Sigma V u^2-\int_{\partial\Sigma}W u^2,
\qquad u\in H^1(\Sigma).
\]
If $u$ is $L_0$-harmonic in $\Sigma$, then integration by parts yields
\begin{equation}\label{eq:Q-boundary}
Q[u]=\int_{\partial\Sigma}u(\partial_\nu u-Wu)=\int_{\partial\Sigma}u\,B_0u.
\end{equation}
In particular, if $u$ solves \eqref{eq:Steklov}, then
$Q[u]=\sigma\int_{\partial\Sigma}u^2$.

Using harmonic extensions $\widehat h$ from Proposition~\ref{prop:Dirichlet-wellposed}, we obtain
the Rayleigh--Ritz characterizations
\begin{equation}\label{eq:sigma0-var}
\sigma_0
=\inf_{h\in H^{1/2}(\partial\Sigma)\setminus\{0\}}
\frac{Q[\widehat h]}{\int_{\partial\Sigma}h^2},
\end{equation}
and, if $h_0$ is the boundary trace of a first Steklov eigenfunction,
\begin{equation}\label{eq:sigma1-var}
\sigma_1
=\inf\left\{\frac{Q[\widehat h]}{\int_{\partial\Sigma}h^2}:\ h\in H^{1/2}(\partial\Sigma)\setminus\{0\},
\ \int_{\partial\Sigma}h\,h_0=0\right\}.
\end{equation}

\begin{remark}\label{rem:DtN-domain}
Without assuming $\lambda_0^D(L_0)>0$, one can still define a Dirichlet-to-Neumann map for $L_0$ on a suitable closed subspace of $L^2(\partial\Sigma)$ (and hence a corresponding Steklov spectrum), 
exactly as in Tran~\cite{Tran};
see also Zhu~\cite[Section 6]{JZhu} for a closely related construction via a Dirichlet-to-Robin map (and the associated modified Steklov spectrum) in the presence of Dirichlet kernel. 
We restrict to the coercive regime to avoid introducing these domain restrictions.
\end{remark}

\subsection{An upper bound for the first Steklov eigenvalue in the curvature-corrected case}\label{subsec:sigma0}

We return to the curvature-corrected pair
\[
L=\Delta+(V-aK),
\qquad
B=\partial_\nu-W+a\kappa_g,
\]
and assume $\lambda_0^D(L)>0$.

\begin{theorem}\label{thm:sigma0}
Let $(\Sigma^2,ds^2)$ be a compact oriented surface with boundary and assume $\lambda_0^D(L)>0$.
Then
\[
\sigma_0
\le
-\frac{1}{|\partial\Sigma|}\int_\Sigma V
-\frac{1}{|\partial\Sigma|}\int_{\partial\Sigma}W
+\frac{2\pi a\,\chi(\Sigma)}{|\partial\Sigma|}.
\]
Consequently,
\begin{equation}\label{eq:sigma0-essinf}
\sigma_0
\le
-\frac{|\Sigma|}{|\partial\Sigma|}\,\essinf_\Sigma V
-\essinf_{\partial\Sigma}W
+\frac{2\pi a\,\chi(\Sigma)}{|\partial\Sigma|}.
\end{equation}
Moreover, if $a\neq 0$ and equality holds in \eqref{eq:sigma0-essinf}, then $V$ and $W$ are constant
a.e.\ and
\[
K=\frac{V}{a},
\qquad
\kappa_g=\frac{\sigma_0+W}{a}
\]
are constant.
\end{theorem}

\begin{proof}
Let $h\equiv 1$ on $\partial\Sigma$ and let $\widehat 1$ be its $L$-harmonic extension. By
\eqref{eq:sigma0-var},
\[
\sigma_0\le \frac{Q[\widehat 1]}{\int_{\partial\Sigma}1}=\frac{Q[\widehat 1]}{|\partial\Sigma|}.
\]
By Proposition~\ref{prop:Dirichlet-wellposed} applied to $L$, the harmonic extension $\widehat 1$
minimizes the interior energy $Q^D$ among all extensions with trace $1$, hence
$Q^D[\widehat 1]\le Q^D[1]$. Since the boundary term in $Q$ is fixed by the trace, we obtain
$Q[\widehat 1]\le Q[1]$, and therefore $\sigma_0\le Q[1]/|\partial\Sigma|$. Finally, by the
definition of $Q$ and Gauss--Bonnet,
\[
Q[1]=-\int_\Sigma(V-aK)-\int_{\partial\Sigma}(W-a\kappa_g)
=-\int_\Sigma V-\int_{\partial\Sigma}W+2\pi a\,\chi(\Sigma),
\]
which yields the first inequality; \eqref{eq:sigma0-essinf} follows from the $\essinf$ bounds.

If equality holds in \eqref{eq:sigma0-essinf}, then $Q[\widehat 1]=Q[1]$, and by uniqueness of the
$L$-harmonic extension this forces $\widehat 1\equiv 1$. Hence $L1=0$ in $\Sigma$ and
$B1=\sigma_0$ on $\partial\Sigma$, implying $V=aK$ and $-W+a\kappa_g=\sigma_0$ a.e.\
If $a\neq 0$, this yields the stated rigidity.
\end{proof}

\subsection{An upper bound for the second Steklov eigenvalue}\label{subsec:sigma1}

We state an estimate in the coercive regime $\lambda_0^D(L_0)>0$, obtained by combining the
Fraser--Schoen \cite{FraserSchoen11} conformal mapping method with the minimizing property of harmonic extensions.
We emphasize that the key conformal input is the same hemisphere mapping used earlier in
Subsubsection~\ref{subsubsec:planar-domains} (cf.\ \eqref{energy-hemisphere}); the difference is
that here the min--max constraint is expressed in terms of boundary traces, so the balancing
condition is imposed on $\partial\Sigma$.

\begin{theorem}\label{thm:sigma1}
Let $(\Sigma^2,ds^2)$ be a compact oriented surface with boundary of genus $\gamma$ and $r$ boundary
components. Assume $\lambda_0^D(L_0)>0$ for $L_0=\Delta+V$, and let $\sigma_1=\sigma_1(L_0,B_0)$ be
the second Steklov eigenvalue of \eqref{eq:Steklov} with $B_0=\partial_\nu-W$. Then
\[
\sigma_1\,|\partial\Sigma|
<
4\pi(\gamma+r)-\int_\Sigma V-\int_{\partial\Sigma}W.
\]
\end{theorem}

\begin{proof}
Let $u_0$ be a first Steklov eigenfunction and set $h_0=u_0|_{\partial\Sigma}$. By Gabard's theorem
\cite{Gabard2006}, there exists a proper conformal branched cover
$\Phi=(\phi_1,\phi_2,\phi_3):\Sigma\to \mathbb S^2_+\subset\mathbb R^3$ of degree at most $\gamma+r$
mapping $\partial\Sigma$ into the equator. In particular,
\[
\sum_{i=1}^3\int_\Sigma|\nabla\phi_i|^2\le 4\pi(\gamma+r),\qquad
\phi_3|_{\partial\Sigma}=0,\qquad \phi_1^2+\phi_2^2=1\ \text{on }\partial\Sigma.
\]
Postcomposing with a conformal automorphism of $\mathbb S^2_+$, we may assume
\[
\int_{\partial\Sigma}\phi_1 h_0=\int_{\partial\Sigma}\phi_2 h_0=0,
\]
by the Hersch balancing argument as in Fraser--Schoen~\cite{FraserSchoen11}.
Let $\widehat\phi_i$ be the $L_0$-harmonic extension of $\phi_i|_{\partial\Sigma}$. Then
$\widehat\phi_i$ is admissible in \eqref{eq:sigma1-var} and, by the minimizing property of harmonic
extensions,
\[
Q[\widehat\phi_i]\le Q[\phi_i],\qquad i=1,2.
\]
Therefore,
\[
\sigma_1\int_{\partial\Sigma}\phi_i^2\le Q[\phi_i],\qquad i=1,2.
\]
Summing and using $\phi_1^2+\phi_2^2=1$ on $\partial\Sigma$ yields
\[
\sigma_1|\partial\Sigma|\le Q[\phi_1]+Q[\phi_2].
\]
Since $\phi_3|_{\partial\Sigma}=0$, the Dirichlet variational principle gives
\[
\int_\Sigma V\phi_3^2\le \int_\Sigma|\nabla\phi_3|^2-\lambda_0^D(L_0)\int_\Sigma\phi_3^2.
\]
Using $\phi_1^2+\phi_2^2=1-\phi_3^2$ on $\Sigma$, one finds
\[
Q[\phi_1]+Q[\phi_2]
\le
\sum_{i=1}^3\int_\Sigma|\nabla\phi_i|^2-\int_\Sigma V-\int_{\partial\Sigma}W
-\lambda_0^D(L_0)\int_\Sigma\phi_3^2.
\]
Combining with the energy bound gives
\[
\sigma_1|\partial\Sigma|+\lambda_0^D(L_0)\int_\Sigma\phi_3^2
\le
4\pi(\gamma+r)-\int_\Sigma V-\int_{\partial\Sigma}W.
\]
Since $\lambda_0^D(L_0)>0$ and $\Phi$ is nonconstant, we have $\phi_3\not\equiv 0$, hence the extra
term is strictly positive, yielding the strict inequality.
\end{proof}

In the proof above we used a branched conformal cover
$\Phi=(\phi_1,\phi_2,\phi_3):\Sigma\to \mathbb S^2_+$ in order to exploit the boundary identity
$\phi_1^2+\phi_2^2=1$ on $\partial\Sigma$, at the expense of the energy bound
$\sum_{i=1}^3\int_\Sigma|\nabla\phi_i|^2\le 4\pi(\gamma+r)$.
If one assumes in addition that $V\le 0$ on $\Sigma$, one can instead work directly with
Gabard's branched cover $\varPhi=(\varphi^1,\varphi^2):\Sigma\to\mathbb D$ of degree at most
$\gamma+r$. In this case $(\varphi^1)^2+(\varphi^2)^2=1$ on $\partial\Sigma$ and
$(\varphi^1)^2+(\varphi^2)^2\le 1$ on $\Sigma$, and the nonpositivity of $V$ allows the interior
deficit $1-(\varphi^1)^2-(\varphi^2)^2$ to be handled without introducing a third coordinate.
This yields the sharper constant $2\pi(\gamma+r)$.

\begin{corollary}\label{cor:sigma1-Vle0}
Let $(\Sigma^2,ds^2)$ be a compact oriented surface with boundary of genus $\gamma$ and $r$ boundary
components. Assume $\lambda_0^D(L_0)>0$ for $L_0=\Delta+V$, and suppose moreover that $V\le 0$ on
$\Sigma$. Let $\sigma_1=\sigma_1(L_0,B_0)$ be the second Steklov eigenvalue of \eqref{eq:Steklov} with
$B_0=\partial_\nu-W$. Then
\[
\sigma_1\,|\partial\Sigma|
\le
2\pi(\gamma+r)-\int_\Sigma V-\int_{\partial\Sigma}W.
\]
\end{corollary}

\section*{Acknowledgments}
We thank Ivaldo Nunes and Jo\~ao Henrique Andrade for their interest in this work and for valuable
comments and suggestions.

This work was carried out within the CAPES--Math AmSud project \emph{New Trends in Geometric Analysis}
(CAPES, Grant 88887.985521/2024-00).

R.\,A.\ was partially supported by the Brazilian National Council for Scientific and Technological
Development (CNPq, Grant 151804/2024-9).
M.\,C.\ was partially supported by CNPq (Grants 405468/2021-0 and 311136/2023-0).
V.\,S.\ was partially supported by FAPEAL (Grant E:60030.0000000971/2024).

\section*{Data Availability}
Data availability is not applicable to this article, as no data sets were generated or analyzed in the course of this research.

\section*{Conflict of Interest}
The authors declare that they have no conflict of interest related to this article.

\bibliographystyle{amsplain}
\bibliography{bibliography}

\end{document}